\documentclass[letterpaper, 10 pt, conference]{ieeeconf}
\IEEEoverridecommandlockouts                              
\usepackage{graphics}
\usepackage{graphicx}
\usepackage{epsfig,amssymb}
\usepackage{amsmath}
\usepackage{mathrsfs}
\usepackage{color}
\usepackage{array}
\usepackage{amsfonts,booktabs}
\usepackage{lipsum,amsmath,multicol}
\usepackage{graphicx}
\usepackage{subfigure}
\usepackage{times}
\usepackage{pdfsync}
\usepackage{pgfplots}
\usepackage{pgfplotstable}
\usepackage[subnum]{cases}
\usepackage{filecontents}
\usepackage{multirow}
\usepackage{algorithm}
\usepackage{algpseudocode}
\usepackage{float}
\usepackage{setspace}
\newtheorem{theorem}{Theorem}
\newtheorem{lemma}{Lemma}

\newcommand{\beq}{\begin{equation}}
\newcommand{\eeq}{\end{equation}}
\newcommand{\beqa}{\begin{eqnarray}}
\newcommand{\eeqa}{\end{eqnarray}}

\newcommand{\paren}[1]{\left(#1\right)}
\newcommand{\sqparen}[1]{\left[#1\right]}
\newcommand{\brparen}[1]{\left\{#1\right\}}
\newcommand{\field}[1]{\ensuremath{\mathbb{#1}}}
\newcommand{\abs}[1]{\left|#1\right|} 
\newcommand{\N}{\ensuremath{\field{N}}} 
\newcommand{\I}[1]{\ensuremath{\mathsf{1}_{\left\{#1\right\}}}} 
\newcommand{\PRP}[1]{\ensuremath{\mathsf{Pr}\left(#1\right)}} 
\newcommand{\ES}[1]{\ensuremath{\mathsf{E}\left[#1 \right]}} 

\newcommand{\BO}[1]{\ensuremath{O\paren{#1}}}

\renewcommand{\vec}[1]{\ensuremath{\boldsymbol{#1}}}

\newcommand{\logp}[1]{\ensuremath{\log\paren{#1}}}
\newcommand{\e}[1]{\ensuremath{{\rm e}^{\paren{#1}}}}

\newcommand{\ESI}[2]{\ensuremath{\mathsf{E}_{#1}\left[#2 \right]}} 
\newcommand{\CES}[2]{\ensuremath{\mathsf{E}\left[\left. #1\right| #2\right]}} 
\newcommand{\CP}[2]{\ensuremath{\mathsf{Pr}\paren{\left.#1\right|#2}}}
\newcommand{\CD}[3]{\ensuremath{f_{#1}\paren{#2\left|#3\right.}}}
\newcommand{\DF}[2]{\ensuremath{f_{#1}\paren{#2}}}
\newcommand{\DSE}[1]{\ensuremath{\mathsf{H}\sqparen{#1}}}
\newcommand{\CDSE}[2]{\ensuremath{\mathsf{H}\sqparen{#1\left|#2\right.}}}
\newcommand{\MI}[2]{\ensuremath{\mathsf{I}\sqparen{#1;#2}}}
\newcommand{\CMI}[3]{\ensuremath{\mathsf{I}\sqparen{#1;#2\left|#3\right.}}}
\newcommand{\KLD}[3]{\ensuremath{\mathsf{D}^{#1}\sqparen{#2\left\|#3\right.}}}

\title{Privacy of Information Sharing Schemes  in a Cloud-based Multi-sensor Estimation Problem}

\author{ Ehsan Nekouei, Mikael Skoglund and Karl H. Johansson
\thanks{School of electrical engineering, KTH Royal Institute of Technology,   Stockholm, Sweden. {nekouei,skoglund, kallej}@kth.se. This work is supported by the Knut and Alice Wallenberg Foundation, the
Swedish Foundation for Strategic Research, the Swedish Research Council.}}

\begin{document}
\maketitle
\thispagestyle{empty}

\begin{abstract}
In this paper, we consider a multi-sensor estimation problem wherein each sensor collects noisy information about its local process,  which is only observed by that sensor, and a common process, which is simultaneously observed by all sensors. The objective is to assess the privacy level of (the local process of) each sensor while the common process is estimated using cloud computing technology. The privacy level of a sensor is defined as the conditional entropy of its local process given the \emph{shared information} with the cloud. Two \emph{information sharing schemes} are considered: a local scheme, and a global scheme. Under the local scheme, each sensor estimates the common process based on its the measurement and transmits its estimate to a cloud. Under the global scheme, the cloud receives the sum of sensors' measurements.  It is shown that, in the local scheme, the privacy level of each sensor is always above  a certain level which is characterized using Shannon's mutual information. It is also proved  that this result becomes tight as the number of sensors increases. We also show that the global scheme is asymptotically private, \emph{i.e.}, the privacy loss of the global scheme decreases to zero at the rate of $\BO{1/M}$ where $M$ is the number of sensors.
\end{abstract}

\section{Introduction}
\subsection{Motivation}
Networked control systems (NCSs) are revolutionizing  our society by enabling invaluable services such as intelligent transportation, smart grids, and smart energy management systems. Complex algorithms, \emph{e.g.,} estimation, control and optimization algorithms, are among the core building blocks of any NCS, and the successful operation of a NCS heavily depends on the performance of these algorithms. However, the algorithms typically demand large amounts of storage and computational capacities. Cloud computing technology provides a low cost, reliable, and flexible solution for the computation and storage requirements of NCSs \cite{VC10}. For example, it enables on-demand computational and storage services and allow the system operator to access the system's information at any geographical location. The high degree of connectivity of NCSs makes them easily adaptable to cloud-based services.

To perform cloud-based services, the required information for accomplishing the task has to be shared with an abstract entity, hereafter, simply called the ``cloud". However, the information sharing procedure  might result in the leakage of \emph{private information}. Especially in NCSs, sensors typically measure multiple correlated processes and some of them might carry private information.  
Thus, from the designer's  point of view, it is crucial to obtain a deep understanding of the potential privacy loss due to sharing information with the cloud. In what follows,  by an \emph{information sharing scheme} we mean a certain rule which determines how sensors' measurements are shared with the cloud. 

In this paper, we consider a cloud-based multi-sensor estimation problem and investigate the following research question:  Given an information sharing scheme, to what extent can the cloud infer about the private information of the sensors?

\subsection{Contributions}
We consider a multi-sensor estimation problem wherein the measurement of each sensor contains noisy information about its local random process, only observed by that sensor, and a common random process, observed by all sensors. The local process carries private information about the local environment of that sensor. The common process is estimated in a cloud using the sensors' measurements. We study the leakage of sensors' private information under two information sharing schemes: a local scheme, and a global scheme. In the local scheme, each sensor first estimates the common process using its own measurement, and then transmits its estimate of the common process to the cloud. In the global scheme, sensors simultaneously transmit their measurements to the cloud. 

 Under each scheme, the privacy level of a sensor is defined as the \emph{conditional entropy} of its local process given the received information by the cloud. 
 In the local scheme, a lower bound on the privacy level of each sensor is derived. It is shown to depend on the mutual information between the input and outputs of a certain model (see the discussion after Lemma~ \ref{Lem: MI-B} for more details). This result indicates that the privacy level of each sensor, in the local scheme, is always above a certain level regardless of the number of sensors. It is shown that the lower bound on the privacy level of sensors in the local scheme becomes tight as the number of sensors increases. In addition our results on the global scheme indicate that it is asymptotically private, \emph{i.e.,} the privacy level of each sensor converges to its maximum privacy level as the number of sensors becomes large. The convergence rate of the privacy level with the number of sensors is also characterized.
\begin{figure*}[!htp]
\centering{\includegraphics[scale=0.5]{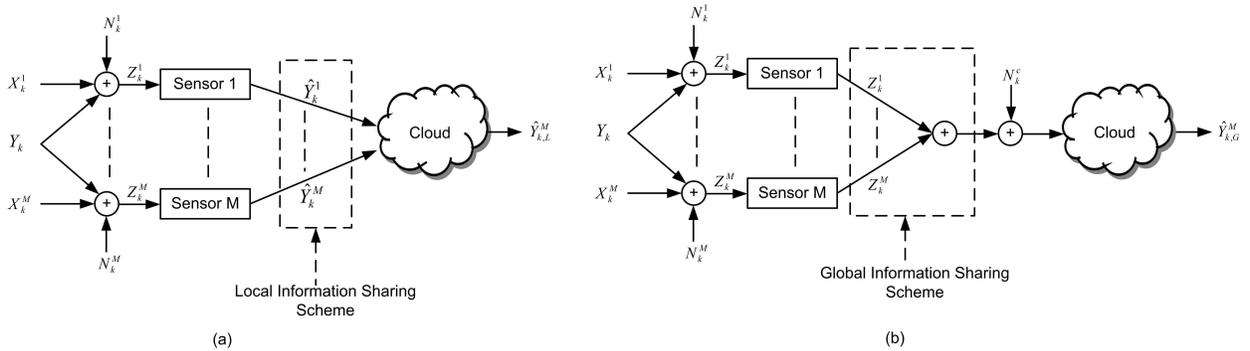}}
\caption{Cloud-based multi-sensor estimation with local $(a)$ and global $(b)$ information sharing schemes.}\label{Fig: Glocal}
\end{figure*}
\subsection{Related Work}
 In  \cite{HTS16, ST16, HT17}, the authors considered a learning-based binary hypothesis testing for a  set-up  in which a group of sensors simultaneously observe a binary private hypothesis and a binary public hypothesis. They proposed various privacy preserving schemes, \emph{e.g.,} linear precoding in \cite{HTS16}, randomized decision rules in \cite{ST16} and a multilayer sensor network in \cite{HT17},  for minimizing the empirical risk of mis-classifying the public hypothesis at a fusion center subject to a constraint on the empirical risk of mis-classifying the private hypothesis by the fusion center.

In \cite{LSTC16}, the authors considered a binary hypothesis test problem with a private hypothesis. They studied the optimal randomized privacy mechanisms for maximizing the type-II error exponent subject to privacy constraints. Li and Oechtering in \cite{LO15}  considered a sensor network in which sensors observe a private binary hypothesis and an eavesdropper intercepts the local decisions of a set of sensors.  They studied the problem of minimizing the Bayes risk of detecting the private hypothesis at a fusion center subject to a privacy constraint at the eavesdropper. The  privacy of the Neyman-Pearson test under a similar set-up was studied in  \cite{LO17}. 

The privacy aspect of estimation problems was considered in \cite{AAL16} and \cite{SDT15}. The authors in \cite{AAL16} studied the minimum mean square estimation of a public random variable subject to a privacy requirement on the estimation error of  a (correlated) private random variable. Sandberg \emph{et al.} \cite{SDT15} considered the state estimation problem in a distribution electricity network subject to  differential privacy constraints for the consumers. 

The authors in \cite{CF12} used the notion of self-information cost to design optimal randomized privacy filters for improving the privacy of a (private) random variable correlated with a public random variable. The interested reader is referred to \cite{MS15,BWI16,KSK17} and references therein for a detailed investigation of the information theoretic approaches to data privacy problem.

The rest of this paper is structured as follows. Next section presents our system model and modeling assumptions. Our main results on the privacy of the local and global schemes are discussed in Section \ref{Sec: RD}. Section \ref{Sec: NR} presents our numerical results and Section \ref{Sec: Conc} concludes the paper. 

\section{System Model}\label{Sec: SM}
Consider a multi-sensor estimation problem with $M$ sensors in which the measurement of sensor $i\in\left\{1,\dots,M\right\}$ at time $k\in\N$ can be written as 
\begin{eqnarray}
Z^i_k= Y_k+X^i_k+N^i_k
\end{eqnarray}
where $Y_k$ and $X^i_k$ are discrete random variables and $N^i_k$ represents the measurement noise of  sensor $i$ at time $k$. The sequence of random variables  $\left\{Y_k\right\}_k$ represents a common process observed by all sensors whereas $\left\{X^i_k\right\}_k$ is a local process only observed by sensor $i$, \emph{i.e.,} the values of $Y_k$ denote some global events observed by all sensors  while the values of $X^i_k$ represent  some events only in the local environment of sensor $i$.

The support sets of $X^i_k$ and $Y_k$ are denoted by $\mathcal{X}^i=\brparen{x_{i1},\dots,x_{im}}$ and $\mathcal{Y}=\brparen{y_1,\dots,y_n}$, respectively. Without loss of generality, we assume that $\abs{\mathcal{X}^i}=m$ for all $i$ . We assume that $\left\{Y_k\right\}_k$  is a sequence of  independent and identically distributed (i.i.d.) random variables with $p^{\rm y}_j=\PRP{Y_k=y_j}$, and $\left\{X^i_k\right\}_k$ is a sequence of i.i.d. random variables with $p^{\rm x}_{ij}=\PRP{X^i_k=x_{ij}}$ for all $i\in\left\{1,\dots,M\right\}$. For each $i$, $\left\{N^i_k\right\}_k$ is assumed to be a set of i.i.d. random variables. The collection of random variables $\left\{Y_k,X^i_k,N^i_k, i\in\left\{1,\dots,M\right\}\right\}_k$ are assumed to be mutually independent.
\subsubsection{Estimation Problem}
Consider the problem of remote estimation of the common process, \emph{i.e.,} $Y_k$, using an abstract entity named ``cloud" which is assumed to be accessible via a network and have storage/processing capabilities. At each time instance, the cloud receives a function of sensors' measurements via an \emph{information sharing scheme}. Two information sharing schemes are considered for {estimating the common process}: a local scheme, and a global scheme.  Fig. \ref{Fig: Glocal} shows a pictorial representation of the local and global information sharing schemes. Under the local scheme, each sensor $i$ at time $k$ first estimates $Y_k$ using the maximum a posteriori probability (MAP) estimator, \emph{i.e.,}  
\begin{align}
\hat{Y}^i_k=\arg\max_{y\in\mathcal{Y}}\CP{Y_k=y}{Z^i_k=z^i_k}\nonumber
\end{align}
where $z^i_k$ is a realization of the random variable $Z^i_k$ and $\hat{Y}^i_k$ is the estimate of $Y_k$ by sensor $i$. Then, sensor $i$ transmits $\hat{Y}^i_k$ to the cloud. Finally, cloud combines the local estimates of sensors, \emph{i.e.,} $\left\{\hat{Y}^i_k\right\}_{i=1}^M$, to form its estimate of $Y_k$. We use $\hat{Y}^M_{k,\rm L}$ to denote the estimate of $Y_k$ by the cloud under the local scheme. 

In the global scheme, at each time $k$, sensors simultaneously transmit their measurements to the cloud. Then, cloud estimates $Y_k$ by using its received information. The received signal by the cloud at time $k$ under the global scheme can be written as
\begin{align}
Z^{{\rm c},M}_k&=\paren{\sum_{i=1}^MZ^i_k}+N^{\rm c}_k\nonumber
\end{align}
where $Z^{{\rm c},M}_k$ and $N^{\rm c}_k$ denote the received signal by the cloud and the received noise at time $k$, respectively. The estimate of $Y_k$ by the cloud under the global scheme is denoted by $\hat{Y}^M_{k,\rm G}$. We assume that $\left\{N^{\rm c}_k\right\}_k$ is a sequence of  i.i.d. random variables and independent of other processes.
\subsubsection{Privacy Metric}
Let $X$ be a generic discrete random variable. Then, the privacy level of $X$ after observing the  (generic) random variable $Z$ is defined as the conditional entropy of $X$ given $Z$, \emph{i.e.,} $\CDSE{X}{Z}$, which can be written as 
\begin{align}
\CDSE{X}{Z}=&-\ESI{Z}{\sum_{x}\CP{X=x}{Z}\log\CP{X=x}{Z}}\nonumber
\end{align}
where $\CP{X=x}{Z}$ denotes the probability of the event $X=x$ conditioned on the value of the random  variable $Z$.  

 Note that $\CDSE{X}{Z}$ quantifies the ambiguity level of $X$ after observing $Z$. For example, if one can perfectly reconstruct $X$ from $Z$, then we have $\CDSE{X}{Z}=0$ which indicates zero privacy. Since conditioning reduces entropy \cite{CT06}, we have 
\begin{align}
\CDSE{X}{Z}\leq\DSE{X} \nonumber.
\end{align}
Thus, the maximum possible privacy level of $X$ is equal to its discrete entropy.

 The choice of conditional entropy as the privacy metric is motivated by the fact that $\CDSE{X}{Z}$ provides a lower bound on the error probability of estimating $X$ using $Z$. More precisely, according to the Fano inequality \cite{CT06}, we have
 \begin{align}
 \PRP{X\neq\hat{X}\paren{Z}}\geq \frac{\CDSE{X}{Z}-1}{\log\abs{\mathcal{X}}}
 \end{align}
 where $\hat{X}\paren{Z}$ denotes the estimate of $X$ using $Z$ and $\abs{\mathcal{X}}$ denotes the cardinality of the support set of $X$. Thus, a large value of $\CDSE{X}{Z}$ indicates that it is less likely to obtain an accurate estimate of $X$ by observing $Z$.

Under each information sharing scheme, the received information by the cloud depends on the sensors' local processes. This allows the cloud to make inference about the local processes, which are considered as private information of sensors. In this paper, the privacy level of the local process of sensor $i$ at time $k$ is measured by the conditional entropy of $X^i_k$ given the received information by cloud. Thus, our metrics for the privacy level of sensor $i$ under the local and global schemes can be written as  $\CDSE{X^i_k}{\hat{Y}^1_k,\dots,\hat{Y}^M_k}$ and $\CDSE{X^i_k}{Z^{{\rm c},M}_k}$, respectively.
\section{Privacy Analysis of The Local and Global Schemes}\label{Sec: RD}
In this section, the privacy of the global and local information sharing schemes is studied. We start our discussions by investigating the privacy level of the local scheme in the next subsection.  
\subsection{Privacy Level of the Local Scheme}
 Before stating our privacy results in the local scheme, we introduce an auxiliary model between each sensor and the cloud  which is helpful in characterizing the privacy level of the local  scheme. The auxiliary model between sensor $i$ and the cloud takes $X^i_k$ as input and outputs $\paren{\hat{Y}^i_k,Y_k}$ as shown in Fig. \ref{Fig: F1}.
\begin{figure}[!htp]
\centering{\includegraphics[scale=0.9]{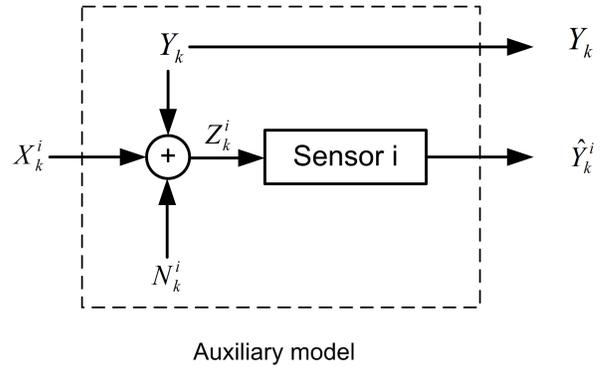}}
\caption{The auxiliary model between sensor $i$ and the cloud.}\label{Fig: F1}
\end{figure}

The next lemma establishes  a lower bound on the privacy level of sensors under the local scheme. 
\begin{lemma}\label{Lem: MI-B}
Let $\CDSE{X^i_k}{\hat{Y}^1_k,\dots,\hat{Y}^M_k}$  denote the privacy level of $X^i_k$ under the local  scheme. Then, we have  
\begin{align}
&\CDSE{X^i_k}{\hat{Y}^1_k,\dots,\hat{Y}^M_k}\geq \DSE{X^i_k}-\MI{X^i_k}{Y_k,\hat{Y}^i_k}
\end{align}
where $\MI{\cdot}{\cdot}$ denotes the Shannon's mutual information.
\end{lemma}
\begin{proof}
See Appendix \ref{App: MI-LB}.
\end{proof}
Lemma \ref{Lem: MI-B} establishes a lower bound on the privacy of the local process of sensor $i$ given the received information by cloud, \emph{i.e.,} $\brparen{\hat{Y}^1_k,\dots,\hat{Y}^M_k}$. The lower bound in this lemma depends on the discrete entropy of $X^i_k$ and the mutual information between the input and outputs of the auxiliary model between sensor $i$ and the cloud. Using Lemma \ref{Lem: MI-B} and the fact that conditioning reduces entropy, we have 
\begin{align}
 \DSE{X^i_k}-\MI{X^i_k}{Y_k,\hat{Y}^i_k}\leq \CDSE{X^i_k}{\hat{Y}^1_k,\dots,\hat{Y}^M_k}\leq  \DSE{X^i_k}\nonumber
\end{align}
Thus, the \emph{privacy loss} of sensor $i$ in the local scheme can at most be equal to the value of mutual information between the input and outputs of the auxiliary  model of sensor $i$. 

 Next, we study the asymptotic behavior of the privacy in the local scheme. To this end, the following assumptions are imposed:
 \begin{enumerate}
 \item The common process is binary valued, \emph{i.e.,} $\mathcal{Y}=\brparen{y_1,y_2}$. 
 \item  The local processes are binary valued and homogeneous, \emph{i.e.,}  $\mathcal{X}^i=\mathcal{X}=\brparen{x_1,x_2}$ and $\PRP{X^i_k=x_1}=\PRP{X^j_k=x_1}$ for $1\leq i,j\leq M$.
 \item  The measurement noises of sensors, \emph{i.e.,} $\left\{N^i_k\right\}_{i=1}^M$, are identically distributed. 
 \end{enumerate}
 
 Let $z^i_k$ denote the measurement of sensor $i$ at time $k$, \emph{i.e.,} $z^i_k$ is a realization of $Z^i_k$. The optimal estimator of ${Y}_k$ at sensor $i$ can be written as  
\begin{align}\label{Eq: MAP-Y}
\hat{Y}^i_k=\arg\max_{y\in\left\{y_1,y_2\right\}}\CP{Y^i_k=y}{Z^i_k=z^i_k}
\end{align}
The next lemma studies the structure of the optimal estimator of $Y_k$ in the cloud under the local scheme.
\begin{lemma}\label{Lem: OE-SL}
Consider the local  scheme under the assumptions 1-3 above. Then, the optimal estimator of $Y_k$ in the cloud can be expressed as 
\begin{align}\label{Eq: Ycl}
\hat{Y}^M_{k,\rm L}=
\left\{
\begin{array}{cc}
y_1,& \text{if}\quad \frac{p^{\rm y}_1p^{M^1_k}\paren{1-p}^{M-M^1_k}}{p^{\rm y}_2{\paren{1-q}^{M^1_k}q^{M-M^1_k}}}\geq 1\nonumber\\
y_2&\quad {\rm Otherwise}\nonumber
\end{array}
\right.
\end{align}
where $p^{\rm y}_1=\PRP{Y_k=y_1}$, $p^{\rm y}_2=\PRP{Y_k=y_2}$, $p=\CP{\hat{Y}^i_k=y_1}{Y_k=y_1}$, $q=\CP{\hat{Y}^i_k=y_2}{Y_k=y_2}$, and $M^1_k=\sum_i\I{\hat{Y}^i_k=y_1}$ is the number of sensors which at time $k$ transmit $y_1$ to the cloud as their estimates of $Y_k$. 
\end{lemma}
\begin{proof}
See Appendix \ref{App: OE-SL}.
\end{proof}

The next lemma derives an upper bound on the error probability of estimating $Y_k$ in the cloud under the local scheme. Later, this upper bound is used to study the privacy level of the local scheme as the number of sensors becomes large.
\begin{lemma}\label{Lem: EB}
Consider the local  scheme under the assumptions 1-3. Then, the error probability of estimating $Y_k$ in the cloud, \emph{i.e.,} $P^{ \rm y}_{\rm L}\paren{M }$, can be upper bounded as 
\begin{align}
P^{ \rm y}_{\rm L}\paren{M }&\leq 2p^{\rm y}_1\exp\paren{-\frac{2M\KLD{2}{p}{1-q}}{\abs{\logp{\frac{q}{1-p}}-\logp{\frac{1-q}{p}}}^2}}\nonumber\\
&+ 2p^{\rm y}_2\exp\paren{-\frac{2M\KLD{2}{1-q}{p}}{\abs{\logp{\frac{q}{1-p}}-\logp{\frac{1-q}{p}}}^2}}
\end{align}
where $\KLD{}{p}{1-q}=p\logp{\frac{p}{1-q}}+\paren{1-p}\logp{\frac{1-p}{q}}$ and $\KLD{}{1-q}{p}=\paren{1-q}\logp{\frac{1-q}{p}}+q\logp{\frac{q}{1-p}}$.
\end{lemma}
\begin{proof}
See Appendix \ref{App: EB}.
\end{proof}
Lemma \ref{Lem: EB} derives an upper bound on the error probability of estimating $Y_k$ in the cloud under the local scheme. This upper bound depends on the number of sensors, $p$, $q$, $p^{\rm y}_1$, $p^{\rm y}_2$ and the Kullback-Leibler (KL) distance between the binary probability distributions $\paren{p,1-p}$ and $\paren{1-q,q}$. Based on this lemma, $P^{ \rm y}_{\rm L}\paren{M }$ decays to zero at least exponentially fast with the number of sensors. 

The next theorem studies the asymptotic behavior of the privacy level under the local scheme with the number of sensors.
\begin{theorem}\label{Theo: MI}
Consider the local scheme under the assumptions 1-3. If $p\neq 1-q$, we have 
\begin{align}
\lim_{M\rightarrow\infty}\CDSE{X^i_k}{\hat{Y}^1_k,\dots,\hat{Y}^M_k}= \DSE{X^i_k}-\MI{X^i_k}{Y_k,\hat{Y}^i_k}
\end{align}
\end{theorem}
\begin{proof}
See Appendix \ref{App: MI}.
\end{proof}
According to Theorem \ref{Theo: MI}, the privacy level of sensor $i$ in the local scheme converges to the difference between the discrete entropy of $X^i_k$ and the mutual information between the input and outputs of the auxiliary model in Fig. \ref{Fig: F1} as the number of sensors grows. 
\subsection{Privacy Level of the Global Scheme}
In this subsection, we study the privacy level of the global information sharing scheme.We assume that $(i)$ the measurement noise of each sensor $i$ is Gaussian distributed with zero mean and variance $\sigma^2_i$, $(ii)$ the received noise in the cloud is Gaussian distributed with zero mean and variance $\sigma^{2}_{\rm c}$. It is also assumed that  we have $0<\sigma^2_{\rm min}=\min\paren{\sigma^2_{\rm c},\inf_i\sigma^2_i}$. 

The next lemma derives a lower bound on the privacy level of the global information sharing scheme.
\begin{lemma}\label{Lem: P-G}
The privacy level of sensor $i$ in the global scheme can be lower bounded as 
\begin{align}
\CDSE{X^i_k}{Z^{{\rm c},M}_k}\geq \DSE{X^i_k}-\frac{\max_{x,x^\prime\in\mathcal{X}^i}\abs{x-x^\prime}^2}{2\paren{M+1}\sigma^2_{\rm min}}
\end{align}
\end{lemma}
\begin{proof}
See Appendix \ref{App: P-G}.
\end{proof}
Lemma \ref{Lem: P-G} establishes a lower bound on the privacy level of sensor $i$ under the global scheme. This lower bound depends on the number of sensors, $\sigma^2_{\rm min}$ and the ``width"  of the support set of $X^i_k$, defined as $\max_{x,x^\prime\in\mathcal{X}^i}\abs{x-x^\prime}$. 

The next theorem studies the behavior of the privacy level of the global scheme when the number of sensors is large. 
\begin{theorem}\label{Theo: PGS}
Let $\CDSE{X^i_k}{Z^{{\rm c},M}_k}$ denote the privacy level of sensor $i$ under the global scheme. Then, we have
\begin{align}
\limsup_{M\rightarrow\infty} M\paren{\DSE{X^i_k}-\CDSE{X^i_k}{Z^{{\rm c},M}_k}}\leq \frac{\max\limits_{x,x^\prime\in\mathcal{X}^i}\abs{x-x^\prime}^2}{2\sigma^2_{\rm min}}.\nonumber
\end{align}
\end{theorem}
\begin{proof}
Using Lemma \ref{Lem: P-G} and the fact that conditioning reduces entropy, the privacy level of sensor $i$ can be upper and lower bounded as 
\begin{align}
\DSE{X^i_k}-\frac{\max_{x,x^\prime\in\mathcal{X}^i}\abs{x-x^\prime}^2}{2\paren{M+1}\sigma^2_{\rm min}}\leq \CDSE{X^i_k}{Z^{{\rm c},M}_k}\leq \DSE{X^i_k}\nonumber
\end{align}
The desired result directly follows from the above inequalities. 
\end{proof}
According to Theorem \ref{Theo: PGS}, the privacy level of $X^i_k$ converges to $\DSE{X^i_k}$, \emph{i.e,} its maximum value, at the rate of $\BO{1/M}$ when the number of sensors becomes large. This observation indicates that the global scheme is asymptotically completely private as the number of sensors increases.

\section{Numerical Results}\label{Sec: NR}
In this section, the privacy of the local and global schemes is numerically evaluated. The local and global processes are assumed to be collections of i.i.d. random variables taking values in $\left\{0,\frac{1}{2}\right\}$. The measurement noise of each sensor $i$ is assumed to be Gaussian distributed with zero mean and variance $\sigma^2_i$.   

Fig. \ref{Fig: L} illustrates the privacy level of sensor 1 under the local and global schemes as a function of the number of sensors. According to Fig. \ref{Fig: L-1}, the privacy level of $X^1_k$ under the local scheme stays above the lower bound provided in Lemma \ref{Lem: MI-B}. Moreover, as the number of sensors becomes large, the privacy level of $X^1_k$ converges to the lower bound in Lemma \ref{Lem: MI-B}, a behavior predicted by Theorem \ref{Theo: MI}. 

Based on Fig. \ref{Fig: G1}, as the number of sensors becomes large, the privacy level of $X^1_k$ under the global scheme, \emph{i.e.,} $\CDSE{X^1_k}{Z^{{\rm c},M}_k}$, converges to the discrete entropy of $X^1_k$, a result established in Lemma \ref{Lem: P-G}. Moreover, as the number of sensors becomes large, it becomes less likely for the cloud to estimate $X^1_k$ correctly under the global scheme. Thus,  the global scheme becomes completely private as the number of sensors increases.  A comparison between Fig. \ref{Fig: L-1} and Fig. \ref{Fig: G1} shows that  the global scheme achieves a higher level of privacy compared with the local scheme when the number of sensors is more than one.  

\begin{figure}[!htp]
\centering
\subfigure[]
{\includegraphics[scale=0.5]{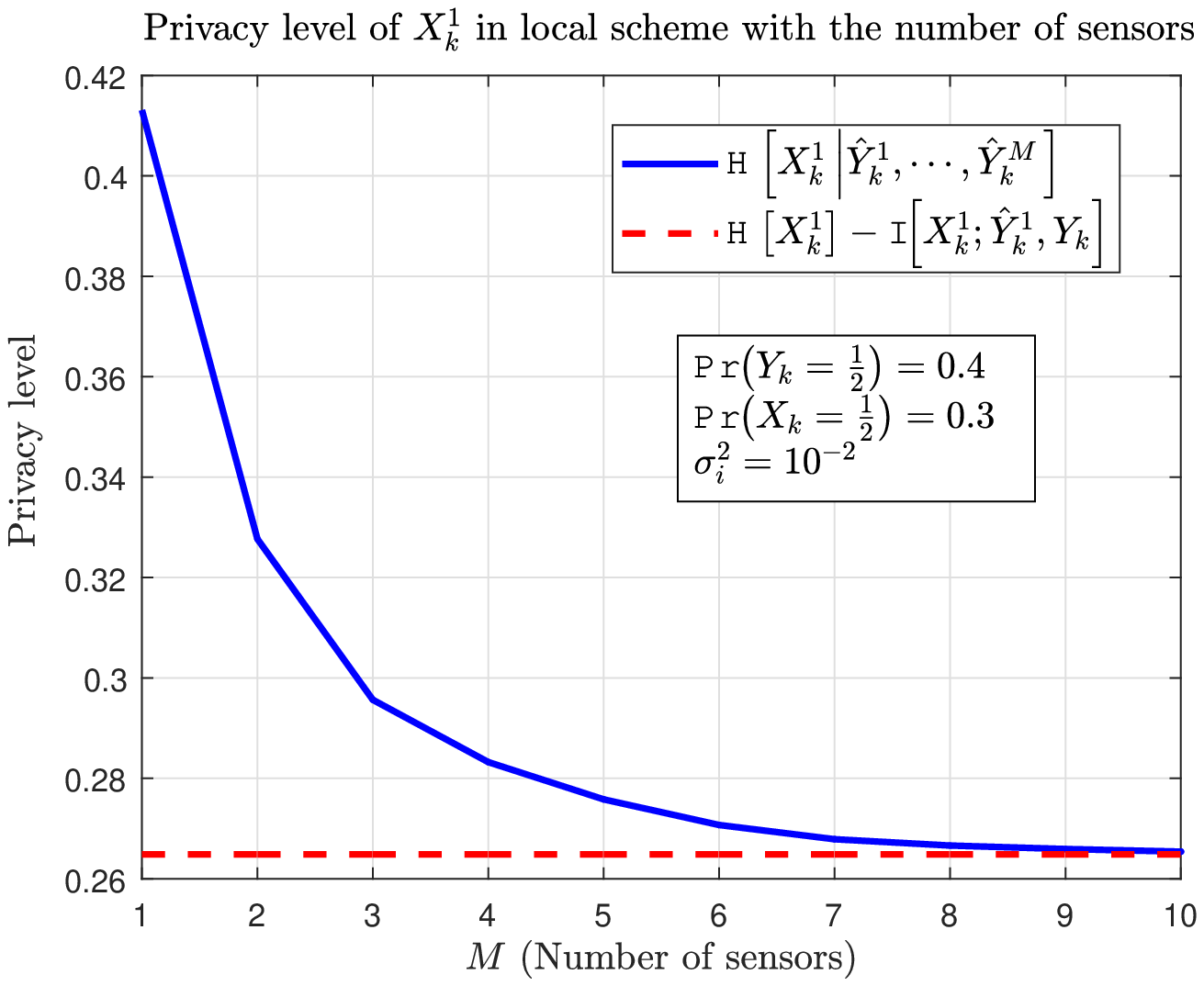}\label{Fig: L-1}}
\subfigure[]
{\includegraphics[scale=0.5]{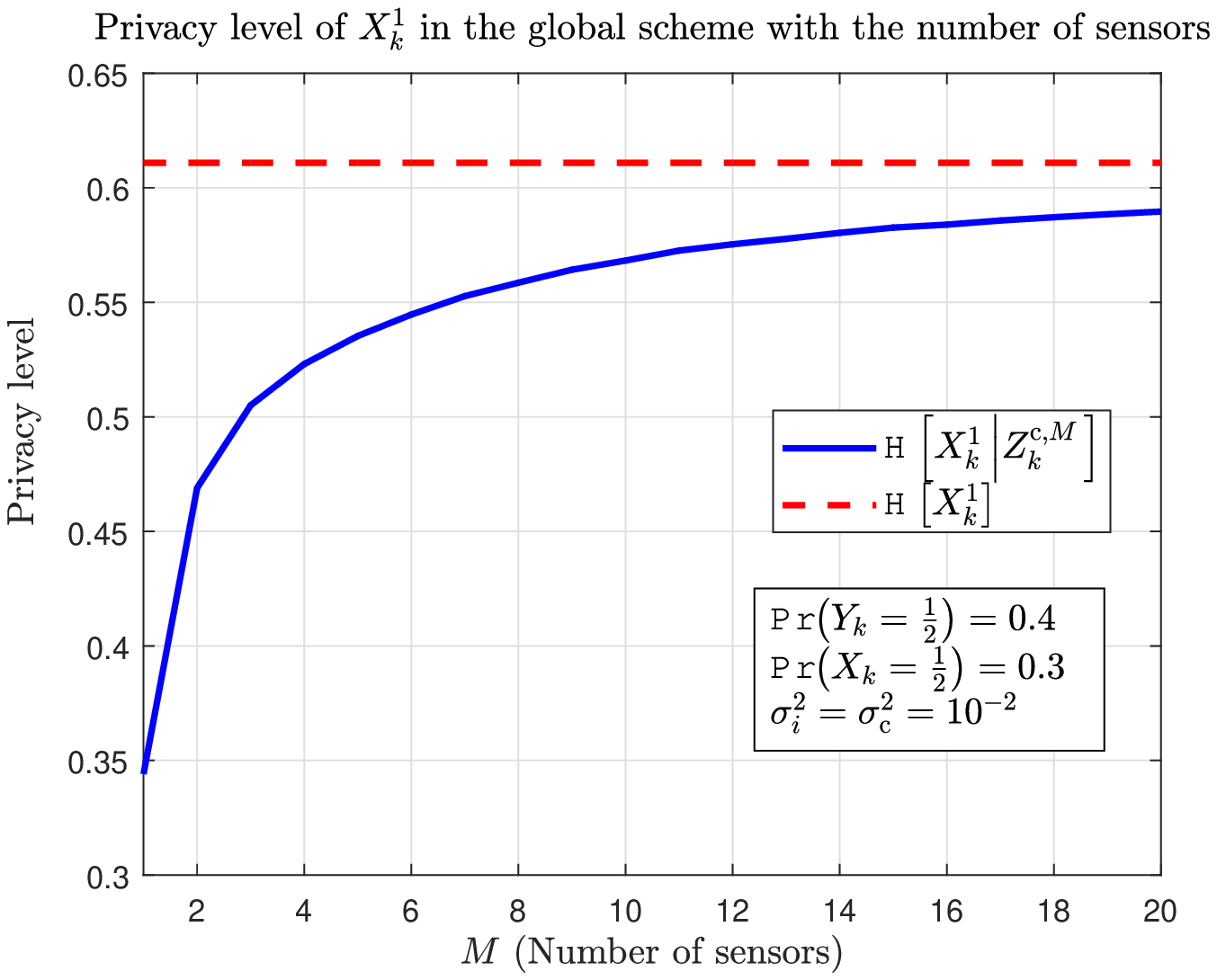}\label{Fig: G1}}
\caption{The privacy level of $X^1_k$ under the local scheme $(a)$ and global scheme $(b)$ with the number of sensors.}\label{Fig: L}
\end{figure}

\section{Conclusions and Future Work}\label{Sec: Conc}
In this paper, we considered a multi-sensor cloud-based estimation problem in which each sensor observes noisy information about its own local process as well as a common process, observed by all sensors. Two information sharing schemes for estimating the common process in a cloud were considered: a local scheme, and a global scheme. The privacy of the local processes of sensors under each information sharing scheme was studied. In particular, it was shown that the privacy level of each sensor in the local scheme is always above a certain level regardless of the number of sensors. It was also shown that the global scheme is asymptotically private.
\appendices
\section{Proof of Lemma \ref{Lem: MI-B}}\label{App: MI-LB}
Using the definition of mutual information, we have \eqref{Eq: MI-Expans}
\begin{figure*}
\begin{align}\label{Eq: MI-Expans}
\DSE{X^i_k}-\CDSE{X^i_k}{\hat{Y}^1_k,\dots,\hat{Y}^M_k}&=\MI{X^i_k}{\hat{Y}^1_k,\dots,\hat{Y}^M_k}\nonumber\\
&\leq\MI{X^i_k}{Y_k,\hat{Y^1_k},\dots,\hat{Y}^M_k}\nonumber\\
&\stackrel{(a)}{=}\MI{X^i_k}{Y_k,\hat{Y^i_k}}+\sum_{j<i}\CMI{X^i_k}{\hat{Y}^j_k}{\hat{Y}^1_k,\dots,\hat{Y}^{j-1}_k,\hat{Y}^i_k,Y_k}\nonumber\\
&+\sum_{j>i}\CMI{X^i_k}{\hat{Y}^j_k}{\hat{Y}^1_k,\dots,\hat{Y}^{j-1}_k,Y_k}
\end{align}
\hrule
\end{figure*}
where $(a)$ follows from the chain rule for mutual information.  Note that given $Y_k$, $\hat{Y}^j_k$ only depends on $N^j_k$ and $X^j_k$ which are independent of $\paren{X^i_k,\hat{Y}^1_k,\dots,\hat{Y}^{j-1}_k,\hat{Y}^{j+1}_k,\dots,\hat{Y}^M_k}$. Thus, the following Markov chains hold: $X^i_k\rightarrow\paren{\hat{Y}^1_k,\dots,\hat{Y}^{j-1}_k,\hat{Y}^i_k,Y_k}\rightarrow \hat{Y}^j_k$ and $X^i_k\rightarrow\paren{\hat{Y}^1_k,\dots,\hat{Y}^{j-1}_k,Y_k}\rightarrow \hat{Y}^j_k$. This implies that  
\begin{align}
&\CMI{X^i_k}{\hat{Y}^j_k}{\hat{Y}^1_k,\dots,\hat{Y}^{j-1}_k,\hat{Y}^i_k,Y_k}=0\nonumber\\
&\CMI{X^i_k}{\hat{Y}^j_k}{\hat{Y}^1_k,\dots,\hat{Y}^{j-1}_k,Y_k}=0\nonumber
\end{align}
which completes the proof.
\section{Proof of Lemma \ref{Lem: OE-SL}}\label{App: OE-SL}
The proof of this lemma is straightforward and is presented here for the sake of clarity. Let $\hat{y}^i_k\in\left\{y_1,y_2\right\}$ denote the received information by cloud from each sensor $i$ in the local scheme. Then, the optimal estimator of $Y_k$ at cloud under the local scheme is given by 
\begin{align}
\hat{Y}^M_{k, \rm L}&=\arg\max_{y\in\left\{y_1,y_2\right\}}\CP{Y_k=y}{\hat{Y}^1_k=\hat{y}^1_k,\dots,\hat{Y}^M_k=\hat{y}^M_k}\nonumber\\
&=\arg\max_{y\in\left\{y_1,y_2\right\}}\CP{\hat{Y}^1_k=\hat{y}^1_k,\dots,\hat{Y}^M_k=\hat{y}^M_k}{Y_k=y}\PRP{Y_k=y}\nonumber\\
&\stackrel{(a)}{=}\arg\max_{y\in\left\{y_1,y_2\right\}}\PRP{Y_k=y}\prod_i\CP{\hat{Y}_k=\hat{y}^i_k}{Y_k=y}\nonumber
\end{align}
where $(a)$ follows from the fact that the random variables $\hat{Y}^1_k,\dots,\hat{Y}^M_k$ are independent of each other conditioned on $Y_k$.
\section{Proof of Lemma \ref{Lem: EB}}\label{App: EB}
To prove this lemma, we consider the following suboptimal estimator for $Y_k$ at cloud
\begin{align}
\tilde{Y}^M_{k}=
\left\{
\begin{array}{cc}
1& \frac{p^{M_1}\paren{1-p}^{M-M_1}}{{\paren{1-q}^{M_1}q^{M-M_1}}}\geq 1\nonumber\\
0&\quad {\rm Otherwise}\nonumber
\end{array}
\right.
\end{align}
Let $E_M$ denote the error event under the suboptimal estimator. Then, we have $P^{ \rm y}_{\rm L}\paren{M }\leq \PRP{E_M}$. The error probability of the suboptimal estimator can be written as \eqref{Eq: SO1}.
\begin{figure*}
\begin{align}\label{Eq: SO1}
\PRP{E_M}&=\CP{E_M}{Y_k=y_1}p^{\rm y}_1+\CP{E_M}{Y_k=y_2}p^{\rm y}_2\nonumber\\
&=\CP{\frac{p^{M_1}\paren{1-p}^{M-M_1}}{{\paren{1-q}^{M_1}q^{M-M_1}}}< 1}{Y_k=y_1}p^{\rm y}_1+\CP{\frac{p^{M_1}\paren{1-p}^{M-M_1}}{{\paren{1-q}^{M_1}q^{M-M_1}}}\geq 1}{Y_k=y_2}p^{\rm y}_2\nonumber\\
&=\CP{\frac{M_1}{M}\logp{\frac{p}{1-q}}+\paren{1-\frac{M_1}{M}}\logp{\frac{1-p}{q}}< 0}{Y_k=y_1}p^{\rm y}_1\nonumber\\
&+\CP{\frac{M_1}{M}\logp{\frac{p}{1-q}}+\paren{1-\frac{M_1}{M}}\logp{\frac{1-p}{q}}\geq 0}{Y_k=y_2}p^{\rm y}_2
\end{align}
\hrule
\end{figure*}
Let $\Phi^i_k=\I{\hat{Y}^i_k=y_1}\logp{\frac{p}{1-q}}+\paren{1-\I{\hat{Y}^i_k=y_1}}\logp{\frac{1-p}{q}}$. Then, we have \eqref{Eq: expans}. 
\begin{figure*}
\begin{align}\label{Eq: expans}
\frac{M_1}{M}\logp{\frac{p}{1-q}}+\paren{1-\frac{M_1}{M}}\logp{\frac{1-p}{q}}&=\frac{1}{M}\sum_i\I{\hat{Y}^i_k=y_1}\logp{\frac{p}{1-q}}+\paren{1-\I{\hat{Y}^i_k=y_1}}\logp{\frac{1-p}{q}}\nonumber\\
&=\frac{1}{M}\sum_i\Phi^i_k
\end{align}
\hrule
\end{figure*}
Note that $\Phi^i_k$ is a discrete random variable taking value from $\left\{\logp{\frac{p}{1-q}},\logp{\frac{1-p}{q}}\right\}$. Also, $\CES{\Phi^i_k}{Y_k=y_1}$ and $\CES{\Phi^i_k}{Y_k=y_2}$ can be written as 
\begin{align}
\CES{\Phi^i_k}{Y_k=y_1}&=p\logp{\frac{p}{1-q}}+\paren{1-p}\logp{\frac{1-p}{q}}\nonumber\\
&=\KLD{}{p}{1-q}
\end{align}
and 
\begin{align}
\CES{\Phi^i_k}{Y_k=y_2}&=\paren{1-q}\logp{\frac{p}{1-q}}+{q}\logp{\frac{1-p}{q}}\nonumber\\
&=-\KLD{}{1-q}{p}
\end{align}
, receptively. Then, we have \eqref{Eq: UB-1}
\begin{figure*}
\begin{align}\label{Eq: UB-1}
\CP{\frac{1}{M}\sum_i\Phi^i_k< 0}{Y_k=y_1}&=\CP{\frac{1}{M}\sum_i\Phi^i_k-\KLD{}{p}{1-q}< -\KLD{}{p}{1-q}}{Y_k=y_1}\nonumber\\
&\leq\CP{\frac{1}{M}\abs{\sum_i\Phi^i_k-\KLD{}{p}{1-q}}> \KLD{}{p}{1-q}}{Y_k=y_1}\nonumber\\
&\stackrel{(a)}{\leq} 2\e{-\frac{2M\KLD{2}{p}{1-q}}{\abs{\logp{\frac{q}{1-p}}-\logp{\frac{1-q}{p}}}^2}}
\end{align}
\hrule
\end{figure*}
where $(a)$ follows from that facts that $\left\{\Phi^i_k\right\}_i$ are conditionally independent given $Y_k$ and the Hoeffding inequality \cite{Gut05}. Similarly, we have  \eqref{Eq: UB-1}
\begin{figure*}
\begin{align}\label{Eq: UB-2}
\CP{\frac{1}{M}\sum_i\Phi^i_k\geq 0}{Y_k=y_2}&=\CP{\frac{1}{M}\sum_i\Phi^i_k+\KLD{}{1-q}{p}\geq \KLD{}{1-q}{p}}{Y_k=y_2}\nonumber\\
&\leq\CP{\frac{1}{M}\abs{\sum_i\Phi^i_k+\KLD{}{1-q}{p}}\geq \KLD{}{1-q}{p}}{Y_k=y_2}\nonumber\\
&{\leq} 2\e{-\frac{2M\KLD{2}{1-q}{p}}{\abs{\logp{\frac{q}{1-p}}-\logp{\frac{1-q}{p}}}^2}}
\end{align}
\hrule
\end{figure*}
which completes the proof.
\section{Proof of Theorem \ref{Theo: MI}}\label{App: MI}
From Lemma \ref{Lem: MI-B}, we have 
\begin{align}\label{Eq: Aux-MI-1}
\liminf_{M\rightarrow\infty}\CDSE{X^i_k}{\hat{Y}^1_k,\dots,\hat{Y}^M_k}\geq \DSE{X^i_k}-\MI{X^i_k}{Y_k,\hat{Y}^i_k}
\end{align}
To prove the other direction, note that the following Markov chain holds: $X^i_k\rightarrow \paren{\hat{Y}^1_k,\dots,\hat{Y}^M_k}\rightarrow\paren{\hat{Y}^i_k,\hat{Y}^M_{k,\rm L}}$ since given $\brparen{\hat{Y}^1_k,\dots,\hat{Y}^M_k}$, the estimate of cloud, \emph{i.e.,}  $\hat{Y}^M_{k,\rm L}$, is known. Thus, we have 
\begin{align}
\DSE{X^i_k}-\CDSE{X^i_k}{\hat{Y}^1_k,\dots,\hat{Y}^M_k}&=\MI{X^i_k}{\hat{Y}^1_k,\dots,\hat{Y}^M_k}\nonumber\\
&\stackrel{(a)}{\geq}\MI{X^i_k}{\hat{Y^i_k},\hat{Y}^M_{k,\rm L}}\nonumber
\end{align}
where $(a)$ follows from the data processing inequality \cite{CT06}. Hence, we have the following upper bound on $\CDSE{X^i_k}{\hat{Y}^1_k,\dots,\hat{Y}^M_k}$ 
\begin{align}\label{Eq: Aux-MI-2}
\CDSE{X^i_k}{\hat{Y}^1_k,\dots,\hat{Y}^M_k}\leq \DSE{X^i_k}-\MI{X^i_k}{\hat{Y}^i_k,\hat{Y}^M_{k,\rm L}}
\end{align}
To complete the proof, we show that $\lim_{M\rightarrow\infty}\MI{X^i_k}{\hat{Y}^i_k,\hat{Y}^M_{k,\rm L}}=\MI{X^i_k}{\hat{Y}^i_k,Y_k}$ as follows. For $\epsilon>0$, we have 
\begin{align}\label{Eq: ASC}
\sum_{M=1}^\infty\PRP{\abs{\hat{Y}^M_{k,\rm L}- Y_k}>\epsilon}&=\sum_{M=1}^\infty\PRP{\hat{Y}^M_{k,\rm L}\neq Y_k}\nonumber\\
&\stackrel{(a)}{<}\infty
\end{align}
where $(a)$ follows from the fact that the error probability of estimating $Y_k$ in the cloud under the local scheme converges to zero exponentially fast with $M$ when $p\neq 1-q$ and assumptions 1-3 hold. From Borel-Cantelli Lemma \cite{PB95} and equation \eqref{Eq: ASC}, we have  $\hat{Y}^M_{k,\rm L}\xrightarrow{a.s.} Y_k$ as $M$ tends to infinity  where \emph{a.s.} stands for almost sure convergence. Following similar steps, it is straightforward to show $\I{X^i_k=x,\hat{Y}^i_k=y,\hat{Y}^M_{k,\rm L}=z}\xrightarrow{a.s.}\I{X^i_k=x,\hat{Y}^i_k=y,Y_k=z}$ for all $x\in\mathcal{X}$ and $y,z\in\mathcal{Y}$. Hence, we have \eqref{Eq: Conv1}
\begin{figure*}[th]
\begin{align}\label{Eq: Conv1}
\lim_{M\rightarrow\infty}\PRP{X^i_k=x,\hat{Y}_k=y,\hat{Y}^M_{k,\rm L}=z}&=\lim_{M\rightarrow\infty}\ES{\I{X^i_k=x,\hat{Y}_k=y,\hat{Y}^M_{k,\rm L}=z}}\nonumber\\
&\stackrel{(b)}{=}\ES{\I{X^i_k=x,\hat{Y}_k=y,Y_k=z}}\nonumber\\
&{=}\PRP{X^i_k=x,\hat{Y}_k=y,Y_k=z}
\end{align}
\hrule
\end{figure*}
where $(b)$ follows from Lebesgue dominated convergence Theorem \cite{PB95}. Following similar steps as above, it is straightforward to show that 
\begin{align}
\lim_{M\rightarrow\infty}\PRP{\hat{Y}^i_k=y,\hat{Y}^M_{k,\rm L}=z}=\PRP{\hat{Y}^i_k=y,Y_k=z}\nonumber
\end{align}
for all $y,z\in\mathcal{Y}$. Using the definition of the mutual information, we have \eqref{Eq: Aux-MI-4}.
\begin{figure*}[th]
\begin{align}\label{Eq: Aux-MI-4}
\lim_{M\rightarrow\infty}\MI{X^i_k}{\hat{Y}^i_k,\hat{Y}^M_{k,\rm L}}&=\lim_{M\rightarrow\infty}\sum_{x\in\mathcal{X},y,z\in\mathcal{Y}}\PRP{X^i_k=x,\hat{Y}^i_k=y,\hat{Y}^M_{k,\rm L}=z}\logp{\frac{\PRP{X^i_k=x,\hat{Y}^i_k=y,\hat{Y}^M_{k,\rm L}=z}}{\PRP{X^i_k=x}\PRP{\hat{Y}^i_k=y,\hat{Y}^M_{k,\rm L}=z}}}\nonumber\\
&=\sum_{x\in\mathcal{X},y,z\in\mathcal{Y}}\PRP{X^i_k=x,\hat{Y}^i_k=y,Y_k=z}\logp{\frac{\PRP{X^i_k=x,\hat{Y}^i_k=y,Y_k=z}}{\PRP{X^i_k=x}\PRP{\hat{Y}^i_k=y,Y_k=z}}}\nonumber\\
&=\MI{X^i_k}{\hat{Y}^i_k,Y_k}
\end{align}
\hrule
\end{figure*}
Combining \eqref{Eq: Aux-MI-2} and \eqref{Eq: Aux-MI-4}, we have 
\begin{align}\label{Eq: Aux-MI-5}
\limsup_{M\rightarrow\infty}\CDSE{X^i_k}{\hat{Y}^1_k,\dots,\hat{Y}^M_k}\leq \DSE{X^i_k}-\MI{X^i_k}{\hat{Y}^i_k,Y_k}
\end{align}
The desired result follows from \eqref{Eq: Aux-MI-1} and \eqref{Eq: Aux-MI-5}.
\section{Proof of Lemma \ref{Lem: P-G}}\label{App: P-G}
Using the definition of mutual information, we have \eqref{Eq: KL1}
\begin{figure*}
\begin{align}\label{Eq: KL1}
\DSE{X^i_k}-\CDSE{X^i_k}{Z^{{\rm c},M}_k}&=\MI{X^i_k}{Z^{{\rm c},M}_k}\nonumber\\
&=\sum_{j=1}^mp^{\rm x}_{ij}\int \CD{Z^{{\rm c}, M}}{z}{X^i_k=x_{ij}}\log\frac{\CD{Z^{{\rm c}, M}}{z}{X^i_k=x_{ij}}}{\DF{Z^{{\rm c}, M}}{z}}dz\nonumber\\
&=\sum_{j=1}^mp^{\rm x}_{ij}\KLD{}{\CD{Z^{{\rm c}, M}}{z}{X^i_k=x_{ij}}}{\DF{Z^{{\rm c}, M}}{z}}
\end{align}
\hrule
\end{figure*}
where $\DF{Z^{{\rm c}, M}}{z}$ and $\CD{Z^{{\rm c}, M}}{z}{A=a}$ denote the density of $Z^{{\rm c},M}_k$  and the conditional density of $Z^{{\rm c},M}_k$ given the event $A=a$, respectively, and $\KLD{}{\cdot}{\cdot}$ denotes the KL distance. The KL term in \eqref{Eq: KL1} can be upper bounded as \eqref{Eq: KL2}
\begin{figure*}
\begin{align}\label{Eq: KL2}
\KLD{}{\CD{Z^{{\rm c}, M}}{z}{X^i_k=x_{ij}}}{\DF{Z^{{\rm c}, M}}{z}}&=\KLD{}{\sum_{j^\prime=1}^mp^{\rm x}_{ij^\prime}\CD{Z^{{\rm c}, M}}{z}{X^i_k=x_{ij}}}{\sum_{j^\prime=1}^Mp^{\rm x}_{ij^\prime}\CD{Z^{{\rm c}, M}}{z}{X^i_k=x_{ij^\prime}}}\nonumber\\
&\stackrel{(a)}{\leq} \sum_{j^\prime=1}^mp^{\rm x}_{ij^\prime}\KLD{}{\CD{Z^{{\rm c}, M}}{z}{X^i_k=x_{ij}}}{\CD{Z^{{\rm c}, M}}{z}{X^i_k=x_{ij^\prime}}}\nonumber\\
&\leq \max_{x,x^\prime\in\mathcal{X}}\KLD{}{\CD{Z^{{\rm c}, M}}{z}{X^i_k=x}}{\CD{Z^{{\rm c}, M}}{z}{X^i_k=x^\prime}}
\end{align} 
\hrule
\end{figure*}
where $(a)$ follows from the convexity of the KL distance. The KL term in the last inequality of \eqref{Eq: KL2} can also be upper bounded as \eqref{Eq: KL3} 
\begin{figure*}
\begin{align}\label{Eq: KL3}
&\KLD{}{\CD{Z^{{\rm c}, M}}{z}{X^i_k=x}}{\CD{Z^{{\rm c}, M}}{z}{X^i_k=x^\prime}}\nonumber\\
&=\KLD{}{\sum_{\vec{x}^{-i}\in\mathcal{X}^{-i},y\in\mathcal{Y}}\hspace{-0.6cm}{\rm P}\paren{\vec{x}^{-i},y}\!\!\CD{Z^{{\rm c}, M}}{z}{X^i_k=x,X^{-i}_k\!=\vec{x}^{-i},Y_k=y}}{\sum_{\vec{x}^{-i}\in\mathcal{X}^{-i},y\in\mathcal{Y}}\hspace{-0.6cm}{\rm P}\paren{\vec{x}^{-i},y}\!\!\CD{Z^{{\rm c}, M}}{z}{X^i_k=x^\prime,X^{-i}_k\!=\vec{x}^{-i},Y_k=y}}\nonumber\\
&\leq\sum_{\vec{x}^{-i}\in\mathcal{X}^{-i},y\in\mathcal{Y}}{\rm P}\paren{\vec{x}^{-i},y} \KLD{}{\CD{Z^{{\rm c}, M}}{z}{X^i_k=x,X^{-i}_k=\vec{x}^{-i},Y_k=y}}{\CD{Z^{{\rm c}, M}}{z}{X^i_k=x^\prime,X^{-i}_k=\vec{x}^{-i},Y_k=y}}\nonumber\\
&\leq\max_{\vec{x}^{-i}\in\mathcal{X}^{-i},y\in\mathcal{Y}} \KLD{}{\CD{Z^{{\rm c}, M}}{z}{X^i_k=x,X^{-i}_k=\vec{x}^{-i},Y_k=y}}{\CD{Z^{{\rm c}, M}}{z}{X^i_k=x^\prime,X^{-i}_k=\vec{x}^{-i},\!Y_k=y}}
\end{align}
\hrule
\end{figure*}
where $\mathcal{X}^{-i}=\prod_{j\neq i}\mathcal{X}^{j}$, $X^{-i}_k$ is the collection of all local processes except the local process of sensor $i$ and ${\rm P}\paren{\vec{x}^{-i},y}=\PRP{X^{-i}_k=\vec{x}^{-i},Y_k=y}$. Note that conditioned on the local and common processes, the received signal by the cloud is a Gaussian random variable. Using the KL distance between two Gaussian random variables, we have \eqref{Eq: KL4}.
\begin{figure*}
\begin{align}\label{Eq: KL4}
\KLD{}{\CD{Z^{{\rm c}, M}}{z}{X^i_k=x,X^{-i}_k=\vec{x}^{-i},Y_k=y}}{\CD{Z^{{\rm c}, M}}{z}{X^i_k=x^\prime,X^{-i}_k=\vec{x}^{-i},Y_k=y}}&=\frac{1}{2}\frac{\abs{x-x^\prime}^2}{\sigma^2_{\rm c}+\sum_i\sigma^2_i}\nonumber\\
&\leq \frac{\max_{x,x^\prime\in\mathcal{X}^i}\abs{x-x^\prime}^2}{2\paren{M+1}\sigma^2_{\rm min}}
\end{align}
\hrule
\end{figure*}
Combining, \eqref{Eq: KL1}-\eqref{Eq: KL4}, we have 
\begin{align}
\CDSE{X^i_k}{Z^{{\rm c},M}_k}\geq \DSE{X^i_k}-\frac{\max_{x,x^\prime\in\mathcal{X}^i}\abs{x-x^\prime}^2}{2\paren{M+1}\sigma^2_{\rm min}}
\end{align}
which completes the proof. 

\bibliographystyle{IEEEtran}
\bibliography{ACC_Privacy}

\end{document}